\documentclass[11pt]{article}
\usepackage{latexsym,amssymb,amsmath,amsthm,enumerate,float,geometry,cite}
\newtheorem{theorem}{Theorem}
\newtheorem{lemma}[theorem]{Lemma}
\newtheorem{proposition}[theorem]{Proposition}
\newtheorem{corollary}[theorem]{Corollary}
\newtheorem{conjecture}[theorem]{Conjecture}

\newtheorem{claim}{Claim}

\usepackage{lineno}
\usepackage{setspace}

\allowdisplaybreaks

\begin{document}
\onehalfspace

\title{Zero-sum copies of spanning forests\\ in zero-sum complete graphs}
\author{Elena Mohr \and Johannes Pardey \and Dieter Rautenbach}
\date{}

\maketitle
\vspace{-10mm}
\begin{center}
{\small Institute of Optimization and Operations Research, Ulm University,\\ 
Ulm, Germany, \texttt{$\{$elena.mohr,johannes.pardey,dieter.rautenbach$\}$@uni-ulm.de}}
\end{center}

\begin{abstract}
For a complete graph $K_n$ of order $n$, 
an edge-labeling $c:E(K_n)\to \{ -1,1\}$ satisfying $c(E(K_n))=0$,
and a spanning forest $F$ of $K_n$, 
we consider the problem to minimize $|c(E(F'))|$
over all isomorphic copies $F'$ of $F$ in $K_n$.
In particular, we ask under which additional conditions
there is a {\it zero-sum} copy, 
that is, a copy $F'$ of $F$ with $c(E(F'))=0$.

We show that there is always a copy $F'$ of $F$ with $|c(E(F'))|\leq \Delta(F)+1$,
where $\Delta(F)$ is the maximum degree of $F$.
We conjecture that this bound can be improved to $|c(E(F'))|\leq (\Delta(F)-1)/2$
and verify this for $F$ being the star $K_{1,n-1}$.
Under some simple necessary divisibility conditions,
we show the existence of a zero-sum $P_3$-factor,
and, for sufficiently large $n$, also of a zero-sum $P_4$-factor.\\[3mm]
\end{abstract}

\section{Introduction}

The kind of zero-sum problem that we study here 
can be traced back to algebraic results 
such as the well-known Erd\H{o}s-Ginzburg-Ziv theorem~\cite{ergizi} 
stating that for every $2m-1$ not necessarily distinct elements 
$a_1,\ldots,a_{2m-1}$ of $\mathbb{Z}_m$,
there is a set $I$ of $m$ indices such that $\sum\limits_{i\in I}a_i$ is $0$.
The two survey articles of Caro~\cite{ca} and Gao and Geroldinger~\cite{gage} 
describe many of the developments and ramifications of this area 
also known as zero-sum Ramsey theory 
within discrete mathematics and additive group theory.
The present paper has been motivated by recent research 
concerning (almost) zero-sum spanning forests 
in edge-labeled complete graphs~\cite{cahalaza,cayu,ehmora,kisi}.
More specifically, we consider a complete graph $K_n$ of order $n$ 
together with a {\it zero-sum $\pm 1$-labeling} of its edges, 
that is, a labeling $c:E(K_n)\to \{ -1,1\}$ that satisfies
$$c(E(K_n))=\sum\limits_{e\in E(K_n)}c(e)=0.$$
For a given spanning forest $F$ of $K_n$, we ask under which additional conditions
there is an isomorphic copy $F'$ of $F$ in $K_n$ that is also {\it zero-sum},
that is, it satisfies $c(E(F'))=0$.

For more general labelings and more general (almost) spanning subgraphs, 
Caro and Yuster~\cite{cayu} recently considered this kind of problem.
In the above setting their results imply the existence of an isomorphic copy $F'$ of $F$ in $K_n$ 
for which $|c(E(F'))|$ is at most $2\Delta(F)$, where $\Delta(F)$ is the maximum degree of $F$.
In~\cite{cahalaza} Caro, Hansberg, Lauri, and Zarb consider $\pm 1$-labelings $c$ of the edges of $K_n$
that are not necessarily zero-sum, 
and provide best a possible upper bound on $|c(E(K_n))|$
that implies the existence of some spanning tree $T$ of $K_n$ that is almost zero-sum,
that is, it satisfies $|c(E(T))|\leq 1$.
Note that this implies that $T$ is zero-sum exactly if $T$ has an even number of edges,
or, equivalently, the order $n$ is odd.
The difference to our question above is that Caro et al.~do not fix the isomorphism type 
of the spanning tree they find.
This is also the case for some well-known results in zero-sum Ramsey theory 
such as F\"{u}redi and Kleitman's~\cite{fukl} proof of Bialostocki's conjecture
that a complete graph $K_n$ 
with integral edge labels
always has a spanning tree whose edge labels sum to $0$ modulo $n-1$, cf.~also~\cite{scse}.
Some results from~\cite{cahalaza} though correspond to special cases of our question.
Caro et al.~provide a best possible upper bound on $|c(E(K_n))|$
that implies the existence of an almost zero-sum spanning path of $K_n$,
that is, an almost zero-sum isomorphic copy of $P_n$ in $K_n$.
In particular, their result implies that a zero-sum complete graph of order $n$, 
where $n-1$ is a multiple of $4$,
has a zero-sum spanning path,
where the divisibility condition on $n$ implies the two clearly necessary conditions
that $K_n$ and $P_n$ both have an even number of edges.
They conclude~\cite{cahalaza} with the open question whether 
every zero-sum $K_n$, where $n$ is a multiple of $4$,
contains a zero-sum perfect matching,
that is, a zero-sum isomorphic copy of $\frac{n}{2}K_2$.
This question was answered affirmatively by Ehard, Mohr, and Rautenbach~\cite{ehmora}
as well as by Kittipassorn and Sinsap~\cite{kisi}.
Again the divisibility condition on $n$ implies the two clearly necessary conditions
that $K_n$ and $\frac{n}{2}K_2$ both have an even number of edges.

The following two simple observations correspond to key arguments for many of the proofs in this area:
\begin{itemize}
\item If $F$ is a spanning forest of $K_n$ and $c$ is a zero-sum $\pm 1$-labeling of the edges of $K_n$,
then, by symmetry, every edge of $K_n$ belongs to the same number of isomorphic copies of $F$ in $K_n$,
which implies that the average of $c(E(F'))$, 
where $F'$ ranges over all isomorphic copies of $F$ in $K_n$, equals $0$.
In particular, there are copies $F^+$ and $F^-$ of $F$ with $c(E(F^+))\geq 0$ and $c(E(F^-))\leq 0$.
\item If $F_1,\ldots,F_k$ are isomorphic copies of $F$ in $K_n$,
$c(E(F_1))\geq 0$, 
$c(E(F_k))\leq 0$, and
each $F_{i+1}$ arises from $F_i$ by removing at most $\ell$ edges and adding at most $\ell$ edges,
then  
$|c(E(F_i))|\leq \ell$
for some $i$.
\end{itemize}
In our setting, these observations yield the following,
which improves the consequence of results from~\cite{cayu} mentioned above.

\begin{proposition}\label{proposition1}
If $c:E(K_n)\to \{ -1,1\}$ is a zero-sum labeling of $K_n$,
and $F$ is a spanning forest of $K_n$,
then there is an isomorphic copy $F'$ of $F$ in $K_n$ with 
$|c(E(F'))|\leq \Delta(F)+1$.
\end{proposition}
For the two special spanning forests $P_n$ and $\frac{n}{2}K_2$ 
of a zero-sum labeled $K_n$,
the additional conditions needed to force zero-sum isomorphic copies 
were just some obviously necessary divisibility conditions.
For many special instances of our general question though, 
the obviously necessary divisibility conditions are far from sufficient
and the key arguments mentioned above are not strong enough to obtain 
the best possible results.
A good example is the star:
If $c:E(K_n)\to \{ -1,1\}$ is a zero-sum labeling of $K_n$,
then a result from Caro and Yuster~\cite{cayu}
implies the existence of a spanning star $K_{1,n-1}$ in $K_n$
with $|c(E(K_{1,n-1}))|\leq 2(n-2)$,
which can be improved to 
$|c(E(K_{1,n-1}))|\leq n$ using Proposition \ref{proposition1}.

Our next result is a best-possible improvement of this.

\begin{theorem}\label{theorem3}
If $c:E(K_n)\to \{ -1,1\}$ is a zero-sum labeling of $K_n$ and $n\geq 2$,
then there is an isomorphic copy $T$ of $K_{1,n-1}$ in $K_n$ 
with $|c(E(T))|\leq \frac{n}{2}-1=\frac{\Delta(K_{1,n-1})-1}{2}$.
\end{theorem}
Theorem \ref{theorem3} is best possible: 
If the positive integer $n$ is a multiple of $4$,
$V(K_n)=\{ u_i:i\in I\}\cup \{ v_i:i\in I\}$ for 
$I=\left\{ 1,\ldots,\frac{n}{2}\right\}$,
and $c:E(K_n)\to \{ -1,1\}$ is such that the graph 
$G$ with vertex set $V(K_n)$ and edge set $c^{-1}(1)$ satisfies
\begin{eqnarray*}
E(G) & = & \big\{ u_iv_j:i,j\in I\mbox{ with $i+j$ even}\big\}
\cup \big\{ u_iu_j:i,j\in I\mbox{ with $i$ and $j$ distinct}\big\},
\end{eqnarray*}
then every copy $T$ of $K_{1,n-1}$ in $K_n$ satisfies $|c(E(T))|=\frac{n}{2}-1$.
After the proof of Theorem \ref{theorem3} in Section \ref{section2},
we describe a similar construction 
for the case that $n-1$ is a multiple of $4$,
which yields a labeling $c:E(K_n)\to \{ -1,1\}$ 
such that every copy $T$ of $K_{1,n-1}$ in $K_n$ 
satisfies $|c(E(T))|\geq \frac{n-5}{2}$.

As a possible strengthening of Theorem \ref{theorem3},
we pose the following conjecture.

\begin{conjecture}\label{conjecture2}
If $c:E(K_n)\to \{ -1,1\}$ is a zero-sum labeling of $K_n$,
and $F$ is a spanning forest of $K_n$,
then there is an isomorphic copy $F'$ of $F$ in $K_n$ with 
$|c(E(F'))|\leq \frac{\Delta(F)-1}{2}$.
\end{conjecture}
Theorem \ref{theorem3} implies the following,
which can be considered a weak version of Conjecture \ref{conjecture2}
for forests of large maximum degree.
\begin{corollary}\label{corollary1}
If $c:E(K_n)\to \{ -1,1\}$ is a zero-sum labeling of $K_n$,
and $F$ is a spanning forest of $K_n$ with $\Delta(F)\geq \frac{n}{2}+1$,
then there is an isomorphic copy $F'$ of $F$ in $K_n$ with 
$|c(E(F'))|\leq \frac{n}{2}-1$.
\end{corollary}
The fact that, for the existence of a zero-sum perfect matching $\frac{n}{2}K_2$, 
the necessary divisibility conditions were sufficient,
inspired us to conjecture the following.

\begin{conjecture}\label{conjecture1}
Let the positive integers $k$ and $n$ be such that 
${n\choose 2}$ and $\frac{(k-1)n}{k}$ are both even integers.
If $T$ is a tree of order $k$, 
$c:E(K_n)\to \{ -1,1\}$ is a zero-sum labeling of $K_n$,
and $n$ is sufficiently large in terms of $k$,
then $K_n$ has a zero-sum $T$-factor,
that is, there is a zero-sum spanning forest $F$ of $K_n$ 
whose components are all isomorphic to $T$.
\end{conjecture}
The divisibility conditions on $k$ and $n$ in Conjecture \ref{conjecture1}
ensure that $K_n$ and a $T$-factor of $K_n$ both have an even number of edges.
For $k=2$, the only possible choice for $T$ is $K_2$, 
that is, Conjecture \ref{conjecture1} concerns a zero-sum perfect matching, 
and it follows from the main results in~\cite{ehmora} and~\cite{kisi}.
We verify Conjecture \ref{conjecture1} for $k=3$,
that is, for $T$ equal to $P_3$, as well as for $T$ equal to $P_4$ and sufficiently large $n$.

\begin{theorem}\label{theorem1}
Let the positive integer $n$ be such that $3$ divides $n$ and ${n\choose 2}$ is even.
If $c:E(K_n)\to \{ -1,1\}$ is a zero-sum labeling of $K_n$,
then $K_n$ has a zero-sum $P_3$-factor.
\end{theorem}

\begin{theorem}\label{theorem2}
Let the positive integer $n$ be such that $8$ divides $n$.
If $c:E(K_n)\to \{ -1,1\}$ is a zero-sum labeling of $K_n$ and $n$ is sufficiently large,
then $K_n$ has a zero-sum $P_4$-factor.
\end{theorem}

All proofs are given in the next section.

\section{Proofs}\label{section2}

We begin with the proof of Proposition \ref{proposition1}.

\begin{proof}[Proof of Proposition \ref{proposition1}]
Clearly, we may assume that $\Delta(F)\geq 1$.
Let $u$ and $v$ be two distinct vertices of $F$.
If one of $u$ or $v$ has degree at most $1$ in $F$,
then removing all edges incident with $u$ or $v$,
and adding all edges between $u$ and $N_F(v)$ as well as all edges between $v$ and $N_F(u)$
results in an isomorphic copy of $F$, in which $u$ and $v$ exchanged their roles.
Note that, for this operation, we removed at most $\Delta(F)+1$ edges
and added at most $\Delta(F)+1$ edges.
Similarly, if $u$ and $v$ both have degree more than $1$ in $F$, 
and $w$ has degree at most $1$ in the forest $F$,
then exchanging the roles 
firstly of $u$ and $w$,
secondly of $u$ and $v$, and
finally of $v$ and $w$ by similar operations as above,
the vertices $u$ and $v$ exchanged their roles,
the vertex $w$ retains its role, 
and we removed and added at most $\Delta(F)+1$ edges for each of the three substeps.
Altogether, these observations imply that we can transform 
every isomorphic copy of $F$ in $K_n$
into every other isomorphic copy of $F$ in $K_n$
by iteratively removing and adding at most $\Delta(F)+1$ edges.
Together with the two simple observations stated before Proposition \ref{proposition1},
the desired result follows.
\end{proof}
The proof idea for Proposition \ref{proposition1} also implies the following statement:
{\it If $c:E(K_n)\to \{ -1,1\}$ is a zero-sum labeling of $K_n$,
and $H$ is a spanning subgraph of $K_n$,
then there is an isomorphic copy $H'$ of $H$ in $K_n$ with 
$|c(E(H'))|\leq \Delta(H)+\delta(H)$, 
where $\delta(H)$ is the minimum degree of $H$.}

We proceed to the proof of our first main result.

\begin{proof}[Proof of Theorem \ref{theorem3}]
Considering the graph $G$ with vertex set $V(K_n)$ and edge set $c^{-1}(1)$
instead of $(K_n,c)$,
it suffices to show that
$$\mbox{\it every graph $G$ of order $n$ 
with exactly $\frac{1}{2}{n\choose 2}$ edges 
has a vertex $u$ with 
$\frac{n}{4}\leq d_G(u)\leq \frac{3}{4}n-1$,}$$
which is easily verified for $n\leq 4$.

Therefore, for a contradiction, 
suppose that $G$ is a graph of order $n$ with $n>4$
that has exactly $m=\frac{1}{2}{n\choose 2}$ edges 
such that there is some $\alpha<\frac{1}{4}$
such that, for every vertex $u$,
either $d_G(u)\leq \alpha n$
or $d_G(u)\geq (1-\alpha)n-1$.
Let 
$V_+=\{ u\in V(G):d_G(u)\geq (1-\alpha)n-1\}$
and $V_-=V(G)\setminus V_+$.
Since $(1-\alpha)n-1>\alpha n$, we have
$V_-=\{ u\in V(G):d_G(u)\leq \alpha n\}.$
We may assume that among all counterexamples,
the graph $G$ is chosen such that 
$$\sum\limits_{u\in V_+}d_G(u)={n\choose 2}-\sum\limits_{u\in V_-}d_G(u)$$
is as large as possible.

Let $n_+=|V_+|$. 
Clearly, $|V_-|=n-n_+$.

If $n_+\geq (1-\alpha)n$, then 
$$\frac{n(n-1)}{2}
=2m
\geq\sum\limits_{u\in V_+}d_G(u)
\geq (1-\alpha)n\big((1-\alpha)n-1\big)
>\frac{3n}{4}\left(\frac{3n}{4}-1\right),$$
which implies $0>\frac{n}{4}\left(\frac{n}{4}-1\right)$
contradicting $n>4$.
Hence, 
$$n_+<(1-\alpha)n,$$
which implies that every vertex in $V_+$ has a neighbor in $V_-$.

If $V_+$ contains two non-adjacent vertices $u$ and $v$, 
and $w$ is a neighbor of $u$ in $V_-$,
then $G'=G-uw+uv$ is a counterexample 
contradicting the choice of $G$.
Hence, 
$$\mbox{$V_+$ is complete.}$$
By symmetry between $G$ and its complement $\bar{G}$,
the choice of $G$ also implies that  
$$\mbox{$V_-$ is independent.}$$
Now, let $d_+$ be the average degree of the vertices in $V_+$,
and let $d_-$ be the average degree of the vertices in $V_-$.
Clearly, $d_+\geq (1-\alpha)n-1>\frac{3}{4}n-1$ and $d_-\leq \alpha n<\frac{1}{4}n$.
The number of edges in $G$ between $V_+$ and $V_-$ equals
\begin{eqnarray}\label{e01}
d_+n_+-n_+(n_+-1)=d_-(n-n_+),
\end{eqnarray}
and the sum of all vertex degrees equals
\begin{eqnarray}\label{e02}
d_+n_++d_-(n-n_+)={n\choose 2}.
\end{eqnarray}
Adding (\ref{e01}) and (\ref{e02}) implies
$\frac{1}{2}n(n-1)+n_+(n_+-1)=2d_+n_+>2\left(\frac{3}{4}n-1\right)n_+$
or, equivalently,
$\left(n_+-\frac{n}{2}\right)\left(n_+-(n-1)\right)>0$,
which implies 
$$n_+<\frac{n}{2}.$$
Subtracting (\ref{e01}) from (\ref{e02}) implies
$\frac{1}{2}n(n-1)-n_+(n_+-1)=2d_-(n-n_+)<\frac{1}{2}n(n-n_+)$
or, equivalently,
$\left(n_+-\frac{n}{2}\right)\left(n_+-1\right)>0$,
which implies 
$$n_+>\frac{n}{2}.$$
The contradiction $\frac{n}{2}<n_+<\frac{n}{2}$ completes the proof.
\end{proof}
If the positive integer $n$ is such that $n-1$ is a multiple of $4$,
$$V(K_n)=\{ u_i:i\in I \}\cup \{ v_i:i\in I \}\cup \{ w\}
\,\,\,\,\,\,\,\,\mbox{ for }\,\,\,\,\,\,\,\,
I=\left\{ 1,\ldots,\frac{n-1}{2}\right\},$$
and $c:E(K_n)\to \{ -1,1\}$ is such that the graph 
$G$ with vertex set $V(K_n)$ and edge set $c^{-1}(1)$ satisfies
\begin{eqnarray*}
E(G) & = & \left\{ u_iv_j:i,j\in I\mbox{ with $i+j$ even}\right\}
\cup \left\{ u_iw:i\in I\right\}
\cup \left\{ v_iw:i\in I\mbox{ with $i$ even}\right\}\\
&& \cup \left(\left\{ u_iu_j:i,j\in I\mbox{ with $i$ and $j$ distinct}\right\}\setminus 
\left\{ u_1u_2,u_3u_4,\ldots,u_{\frac{n-3}{2}}u_{\frac{n-1}{2}}\right\}\right),
\end{eqnarray*}
then a copy $T$ of $K_{1,n-1}$ in $K_n$ centered in $c$ satisfies
$$
|c(E(T))|=
\begin{cases}
\frac{n-5}{2}, & \mbox{ if either $c=u_i$ or $i$ is even and $c=v_i$, and}\\
\frac{n-1}{2}, & \mbox{ if either $c=w$ or $i$ is odd and $c=v_i$}.
\end{cases}
$$
It is not hard to derive Corollary \ref{corollary1} from Theorem \ref{theorem3}.
\begin{proof}[Proof of Corollary \ref{corollary1}]
By Theorem \ref{theorem3} and symmetry,
we may assume the existence of a vertex $u$ of $K_n$
with $\frac{n}{4}\leq d_-\leq d_+\leq \frac{3}{4}n-1$,
where $d_-$ and $d_+$ are the numbers of edges $e$ of $K_n$
that are incident with $u$, and satisfy $c(e)=-1$ and $c(e)=1$ , 
respectively.
Let $v$ be a vertex of $F$ of maximum degree $\Delta(F)$.
We consider an isomorphic copy $F'$ of $F$ in $K_n$, 
where $u$ has the role of $v$.
If $\Delta(F)\geq 2d_-$, 
then we can clearly ensure that $F'$ contains all 
$d_-$ edges $e$ at $u$ with $c(e)=-1$
as well as $d_-$ further edges $e'$ at $u$ with $c(e')=1$,
which implies $|c(E(F'))|\leq \left(n-1-\frac{n}{4}\right)-\frac{n}{4}=\frac{n}{2}-1$.
If $\Delta(F)<2d_-$, 
then we can clearly ensure that $|c(E_u)|\leq 1$,
where $E_u$ is the set of edges of $F'$ 
that are incident with $u$,
which implies $|c(E(F'))|\leq \left(n-1-\Delta(F)\right)+1\leq \frac{n}{2}-1$.
\end{proof}
We proceed to the proof of our further main results.

\begin{proof}[Proof of Theorem \ref{theorem1}]
The proof is by contradiction, that is, 
we assume that $n$ and $c$ have the stated properties
but that $K_n$ has no zero-sum $P_3$-factor.
By Proposition \ref{proposition1},
there is a $P_3$-factor $F$ in $K_n$ with $|c(E(F))|\leq \Delta(P_3)+1=3$.
Since $F$ has an even number of edges, the total sum $c(E(F))$ of the edge labels is even,
which implies $|c(E(F))|\leq 2$.
Possibly considering $-c$ instead of $c$, that is, exchanging the labels $-1$ and $1$,
we may assume that $c(E(F))=2$.

For a path $P:u_1u_2u_3$ in $K_n$ the \textit{type} of $P$ is $\left(c(u_1u_2),c(u_2u_3)\right)$. 
Up to symmetry, there are exactly the three following types:
$$
\big(1,1\big) \,\,\,\, \big(1,-1\big) \,\,\,\, \big(-1,-1\big). 
$$
Since $c(E(F))=2$, at least one path in $F$ is of type $(1,1)$.

If $E_{\rm out}$ and $E_{\rm in}$ are two sets of edges of $K_n$ such that 
$(V(K_n),(E(F)\setminus E_{\rm out})\cup E_{\rm in})$
is a $P_3$-factor $F'$ of $K_n$ with $c(F')=0$,
then we say that the {\it edge-exchange $-E_{\rm out}+E_{\rm in}$} is {\it contradicting}.
Clearly, the existence of a contradicting edge-exchange yields a contradiction.

If $u_1u_2u_3$ is a path of type $(1,1)$ or $(1,-1)$ in $F$,
then, considering the edge exchange $-\{u_1u_2\}+\{u_1u_3\}$,
it follows that $c(u_1u_3)=1$. 
Note that this implies a complete symmetry 
between the vertices of every path of type $(1,1)$ in $F$,
as well as a symmetry between the vertices $u_2$ and $u_3$
of every path $u_1u_2u_3$ of type $(1,-1)$ in $F$.
Later in this proof we exploit another natural symmetry 
based on the fact that $c$ is zero-sum if and only if $-c$ is zero-sum.
In particular, if $F'$ is a $P_3$-factor in $K_n$ with $c(E(F'))=-2$,
then, exchanging the roles of labels $1$ and $-1$,
statements that hold for $F$
also hold in a symmetric version for $F'$.

\begin{claim}\label{P3:claim1a}
If $P:u_1u_2u_3$ is a path of type $(1,1)$ in $F$ and
$R:w_1w_2w_3$ is a path of type $(1,-1)$ in $F$,
then $c(e)=1$ for every edge $e$ between $V(P)$ and $V(R)$.
\end{claim}
\begin{proof}[Proof of Claim \ref{P3:claim1a}]
If $c(u_1w_1)=c(u_3w_3) = -1$, 
then the edge exchange $-\{u_1u_2,w_2w_3\}+\{u_1w_1,u_3w_3\}$ 
is contradicting.
If $c(u_1w_1)=-1$ and $c(u_3w_3) =1$, 
then the edge exchange $-\{u_2u_3,w_1w_2\}+\{u_1w_1,u_3w_3\}$ 
is contradicting. 
See Figure \ref{figee} for an illustration.
If $c(u_1w_1)=1$ and $c(u_3w_3) = -1$, 
then the edge exchange $-\{u_2u_3,w_1w_2\}+\{u_1w_1,u_3w_3\}$ 
is contradicting.
By the symmetry between the vertices of $P$
as well as between $w_2$ and $w_3$, 
it follows that $c(e)=1$ for every edge $e$ between $V(P)$ and $V(R)$.
\end{proof}

\begin{figure}[h]
\begin{center}
{\footnotesize
\unitlength 0.85mm 
\linethickness{0.4pt}
\ifx\plotpoint\undefined\newsavebox{\plotpoint}\fi 
\begin{picture}(178,64)(0,0)
\put(5,5){\circle*{2}}
\put(5,45){\circle*{2}}
\put(135,5){\circle*{2}}
\put(135,45){\circle*{2}}
\put(5,20){\circle*{2}}
\put(5,60){\circle*{2}}
\put(135,20){\circle*{2}}
\put(135,60){\circle*{2}}
\put(25,5){\circle*{2}}
\put(25,45){\circle*{2}}
\put(155,5){\circle*{2}}
\put(155,45){\circle*{2}}
\put(25,20){\circle*{2}}
\put(25,60){\circle*{2}}
\put(155,20){\circle*{2}}
\put(155,60){\circle*{2}}
\put(45,5){\circle*{2}}
\put(45,45){\circle*{2}}
\put(175,5){\circle*{2}}
\put(175,45){\circle*{2}}
\put(45,20){\circle*{2}}
\put(45,60){\circle*{2}}
\put(175,20){\circle*{2}}
\put(175,60){\circle*{2}}
\put(5,5){\line(1,0){20}}
\put(5,45){\line(1,0){20}}
\put(135,45){\line(1,0){20}}
\put(5,20){\line(1,0){20}}
\put(5,60){\line(1,0){20}}
\put(25,5){\line(1,0){20}}
\put(155,5){\line(1,0){20}}
\put(25,45){\line(1,0){20}}
\put(25,20){\line(1,0){20}}
\put(25,60){\line(1,0){20}}
\put(135,20){\line(1,0){20}}
\put(155,60){\line(1,0){20}}
\put(5,24){\makebox(0,0)[cc]{$u_1$}}
\put(5,64){\makebox(0,0)[cc]{$u_1$}}
\put(135,24){\makebox(0,0)[cc]{$u_1$}}
\put(135,64){\makebox(0,0)[cc]{$u_1$}}
\put(5,1){\makebox(0,0)[cc]{$w_1$}}
\put(5,41){\makebox(0,0)[cc]{$w_1$}}
\put(135,1){\makebox(0,0)[cc]{$w_1$}}
\put(135,41){\makebox(0,0)[cc]{$w_1$}}
\put(25,24){\makebox(0,0)[cc]{$u_2$}}
\put(25,64){\makebox(0,0)[cc]{$u_2$}}
\put(155,24){\makebox(0,0)[cc]{$u_2$}}
\put(155,64){\makebox(0,0)[cc]{$u_2$}}
\put(25,1){\makebox(0,0)[cc]{$w_2$}}
\put(25,41){\makebox(0,0)[cc]{$w_2$}}
\put(155,1){\makebox(0,0)[cc]{$w_2$}}
\put(155,41){\makebox(0,0)[cc]{$w_2$}}
\put(45,24){\makebox(0,0)[cc]{$u_3$}}
\put(45,64){\makebox(0,0)[cc]{$u_3$}}
\put(175,24){\makebox(0,0)[cc]{$u_3$}}
\put(175,64){\makebox(0,0)[cc]{$u_3$}}
\put(45,1){\makebox(0,0)[cc]{$w_3$}}
\put(45,41){\makebox(0,0)[cc]{$w_3$}}
\put(175,1){\makebox(0,0)[cc]{$w_3$}}
\put(175,41){\makebox(0,0)[cc]{$w_3$}}
\put(0,20){\makebox(0,0)[cc]{$P$}}
\put(0,60){\makebox(0,0)[cc]{$P$}}
\put(0,5){\makebox(0,0)[cc]{$R$}}
\put(0,45){\makebox(0,0)[cc]{$R$}}
\put(14,22){\makebox(0,0)[cc]{$+$}}
\put(14,62){\makebox(0,0)[cc]{$+$}}
\put(14,3){\makebox(0,0)[cc]{$+$}}
\put(14,43){\makebox(0,0)[cc]{$+$}}
\put(144,43){\makebox(0,0)[cc]{$+$}}
\put(35,22){\makebox(0,0)[cc]{$+$}}
\put(35,62){\makebox(0,0)[cc]{$+$}}
\put(145,22){\makebox(0,0)[cc]{$+$}}
\put(165,62){\makebox(0,0)[cc]{$+$}}
\put(35,3){\makebox(0,0)[cc]{$-$}}
\put(165,3){\makebox(0,0)[cc]{$-$}}
\put(35,43){\makebox(0,0)[cc]{$-$}}
\put(5,20){\line(0,-1){15}}
\put(5,60){\line(0,-1){15}}
\put(135,20){\line(0,-1){15}}
\put(135,60){\line(0,-1){15}}
\put(45,20){\line(0,-1){15}}
\put(45,60){\line(0,-1){15}}
\put(175,20){\line(0,-1){15}}
\put(175,60){\line(0,-1){15}}
\put(45,5){\line(0,1){0}}
\put(45,45){\line(0,1){0}}
\put(175,5){\line(0,1){0}}
\put(175,45){\line(0,1){0}}
\put(2,12){\makebox(0,0)[cc]{$-$}}
\put(2,52){\makebox(0,0)[cc]{$-$}}
\put(48,12){\makebox(0,0)[cc]{$+$}}
\put(48,52){\makebox(0,0)[cc]{$-$}}
\put(132,12){\makebox(0,0)[cc]{$-$}}
\put(132,52){\makebox(0,0)[cc]{$-$}}
\put(178,12){\makebox(0,0)[cc]{$+$}}
\put(178,52){\makebox(0,0)[cc]{$-$}}
\put(90,15){\makebox(0,0)[cc]{$-\{ u_2u_3,w_1w_2\}+\{ u_1w_1,u_3w_3\}$}}
\put(90,55){\makebox(0,0)[cc]{$-\{ u_1u_2,w_2w_3\}+\{ u_1w_1,u_3w_3\}$}}
\put(55,12){\vector(1,0){70}}
\put(55,52){\vector(1,0){70}}
\end{picture}
}
\end{center}
\caption{Two contradicting edge-exchanges.}\label{figee}
\end{figure}
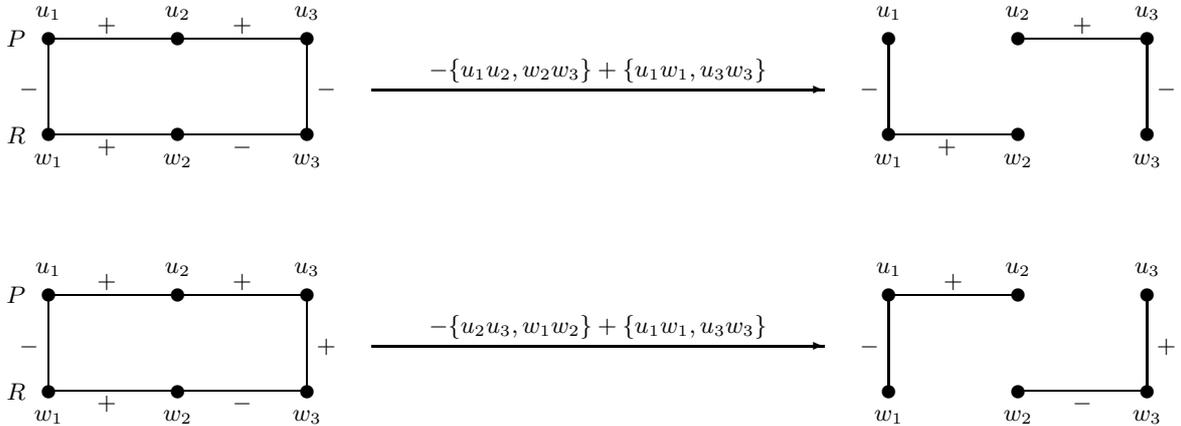

\begin{claim}\label{P3:claim1b}
If $P:u_1u_2u_3$ is a path of type $(1,1)$ in $F$ and
$S:x_1x_2x_3$ is a path of type $(-1,-1)$ in $F$,
then $c(e)=1$ 
for at least six of the nine edges $e$ 
between $V(P)$ and $V(S)$.
\end{claim}
\begin{proof}[Proof of Claim \ref{P3:claim1b}]
If $c(u_1x_1)=c(u_3x_3) = -1$, 
then the edge exchange $-\{u_1u_2,x_2x_3\}+\{u_1x_1,u_3x_3\}$ 
is contradicting. 
If $c(u_1x_1) =c(u_2x_2) = -1$, 
then the edge exchange $-\{u_1u_2,u_2u_3,x_1x_2\}+\{u_1x_1,u_1u_3,u_2x_2\}$ 
is contradicting. 
By symmetry, this implies the existence of a vertex $x$ on $S$
such that every edge $e$ with $c(e)=-1$ 
between $V(P)$ and $V(S)$
is incident with $x$,
which clearly implies the desired statement.
\end{proof}

\begin{claim}\label{P3:claim1c}
If $R:w_1w_2w_3$ and $T:y_1y_2y_3$ are two distinct paths of type $(1,-1)$ in $F$,
then $c(e)=1$ for at least five of the nine edges $e$ between $V(R)$ and $V(T)$.
\end{claim}
\begin{proof}[Proof of Claim \ref{P3:claim1c}]
First, we assume that $c(w_1y_1)=-1$.
If $c(w_3y_3) = -1$, 
then the edge exchange $-\{ w_2w_3,y_1y_2\}+\{ w_1y_1,w_3y_3\}$ 
is contradicting. 
If $c(w_1y_2) = -1$ and $c(w_3y_3) = 1$, 
then the edge exchange $-\{ w_1w_2,y_1y_2,y_2y_3\}+\{ w_1y_1,w_1y_2,w_3y_3\}$ 
is contradicting. 
Hence, by symmetry, we obtain that $c(e)=1$
for the four edges $e$ between $\{ w_2,w_3\}$ and $\{ y_2,y_3\}$
as well as for $e=w_1y_2$,
which yields the desired five edges between $V(R)$ and $V(T)$ in this case.

Next, we assume that $c(w_1y_1)=1$.
If $c(w_1y_2)=c(w_2y_3) = -1$, 
then the edge exchange $-\{ w_1w_2,y_1y_2,y_2y_3\}+\{ w_1y_1,w_1y_2,w_2y_3\}$ is contradicting. 
Hence, $c(e)=1$ for at least one edge $e$ in
 $\{ w_1y_2,w_2y_3\}$.
Symmetric arguments imply that 
$c(e)=1$ for at least one edge $e$ in 
each of the three sets 
$\{ w_1y_3,w_3y_2\}$,
$\{ w_2y_1,w_3y_3\}$, and
$\{ w_3y_1,w_2y_2\}$.
Again, together with $w_1w_2$,
this yields the desired five edges between $V(R)$ and $V(T)$ also in this case.
\end{proof}

\begin{claim}\label{P3:claim1}
There are two paths $P$ and $Q$ of type $(1,1)$ in $F$
such that $c(e)=-1$ for some edge $e$ of $K_n$ between $V(P)$ and $V(Q)$.
\end{claim}
\begin{proof}[Proof of Claim \ref{P3:claim1}]
Let 
\begin{itemize}
\item $a$ be the number of paths of type $(1,1)$ in $F$,
\item $b$ be the number of paths of type $(1,-1)$ in $F$, and
\item $c$ be the number of paths of type $(-1,-1)$ in $F$.
\end{itemize}
Suppose, for a contradiction, 
that two paths $P$ and $Q$ of type $(1,1)$ in $F$
together with an edge $e$ with $c(e)=-1$ between $V(P)$ and $V(Q)$ do not exist.
This implies that the nine edges between every two paths of type $(1,1)$ in $F$
contribute $9$ to $c(E(K_n))$.
By Claim \ref{P3:claim1a},
the nine edges between a path of type $(1,1)$ in $F$
and a path of type $(1,-1)$ in $F$ 
also contribute $9$ to $c(E(K_n))$.
By Claim \ref{P3:claim1b},
the nine edges between a path of type $(1,1)$ in $F$
and a path of type $(-1,-1)$ in $F$ 
contribute at least $6-3$ to $c(E(K_n))$.
By Claim \ref{P3:claim1c},
the nine edges between every two paths of type $(1,-1)$ in $F$
contribute at least $5-4$ to $c(E(K_n))$.

Since $c(E(F))=2$, we have $a=c+1$, and, hence,
\begin{align*}
    c(E(K_n)) &\geq 9\binom{a}{2} + 9ab  + 3a + (5-4)\binom{b}{2} +(2-1)b + (6-3)ac\\
    &- 9\binom{c}{2} - 9bc -3c\\
    & > 0,
\end{align*}
which is a contradiction.
\end{proof}
By Claim \ref{P3:claim1} and symmetry,
we may assume that $P:u_1u_2u_3$ and $Q:v_1v_2v_3$ are two paths 
of type $(1,1)$ in $F$ such that $c(u_1v_1)=-1$.
Since $F$ contains the two paths $P$ and $Q$ of type $(1,1)$,
it also contains at least one path $S:x_1x_2x_3$ of type $(-1,-1)$.
If $c(u_3v_3) =1$, 
then the edge exchange $-\{u_1u_2,v_2v_3\}+\{ u_1v_1,u_3v_3\}$ 
is contradicting. 
By symmetry, this implies $c(e)=-1$ for every edge $e$ 
between $V(P)$ and $V(Q)$.

The $P_3$-factor $F'=\big(V(K_n),(E(F)\setminus \{u_1u_2,v_2v_3\})\cup \{ u_1v_1,u_3v_3\}\big)$
satisfies $c(F')=-2$.

If $R:w_1w_2w_3$ is a path of type $(1,-1)$ in $F$,
and hence, by construction, also in $F'$,
then exchanging the roles of $1$ and $-1$, 
and arguing similarly as before Claim \ref{P3:claim1a},
we obtain the contradiction $c(w_1w_3)=-1$.
Hence, $F$ contains no path of type $(1,-1)$.

If $F$ contains a path $P'$ of type $(1,1)$ 
that is distinct from $P$ and $Q$,
then, by Claim \ref{P3:claim1b},
$c(e)=1$ for at least six of the nine edges $e$ 
between $V(P')$ and $V(S)$.
Now, considering $F'$, 
and repeating the symmetric argument as in the proof of Claim \ref{P3:claim1b}
exchanging the roles of $1$ and $-1$,
implies that $c(e)=-1$ for at least six of the nine edges $e$ 
between $V(P')$ and $V(S)$.
This contradiction implies that such a path $P'$ does not exist,
that is, $F$ consists of the two paths $P$ and $Q$ of type $(1,1)$
and the one path $S$ of type $(-1,-1)$.
Consequently, $F'$ consists of 
one path of type $(-1,-1)$
and two paths of type $(1,-1)$.
Now, 
arguing symmetrically as in Claim \ref{P3:claim1},
exchanging the roles of $1$ and $-1$,
we obtain the contradiction 
that $F'$ contains two paths of type $(-1,-1)$
together with some edge $e$ with $c(e)=1$ between them.
This contradiction completes the proof.
\end{proof}
For the proof of Theorem \ref{theorem2}, 
we need the following technical lemma.

\begin{lemma}\label{lemma1}
If $(x_1,x_2,x_3,x_4,x_5)\in [0,1]^5$ is such that 
$$x_1+x_2+x_3+x_4+x_5=1\,\,\,\,\,\,\mbox{ and }\,\,\,\,\,\,3x_1+x_2-x_3-x_4-3x_5 \geq 0,$$
then $f(x_1,x_2,x_3,x_4,x_5)\geq \frac{5}{256}$, where
$$f(x_1,x_2,x_3,x_4,x_5)=\frac{x_1^2}{2}+\frac{x_2^2}{2}+\frac{x_3^2}{2}-\frac{3x_4^2}{16}-\frac{x_5^2}{2}
+x_1x_2+x_1x_3+x_2x_3+\frac{x_1x_4}{2}+\frac{x_1x_5}{2}-\frac{x_4x_5}{2}.$$
\end{lemma}
\begin{proof}
Since $f$ is smooth and the constraints describe a compact subset of $\mathbb{R}^5$,
we may assume that $(x_1,x_2,x_3,x_4,x_5)$ minimizes the value of $f$ subject to the constraints.

If $x_3>0$, then, for some sufficiently small $\epsilon>0$, 
the vector $(x_1,x_2,x_3-\epsilon,x_4+\epsilon,x_5)$ satisfies the constraints, and
\begin{eqnarray*}
f(x_1,x_2,x_3-\epsilon,x_4+\epsilon,x_5)-f(x_1,x_2,x_3,x_4,x_5)
&=& -\left(\frac{x_1}{2}+x_2+x_3+\frac{3x_4}{8}+\frac{x_5}{2}\right)\epsilon+\frac{5}{16}\epsilon^2\\
& \leq & -\frac{3}{8}\epsilon+\frac{5}{16}\epsilon^2 < 0 
\end{eqnarray*}
yields a contradiction to the choice of $(x_1,x_2,x_3,x_4,x_5)$.
Hence, we have $$x_3=0.$$
If $3x_1+x_2-x_4-3x_5>0$, then either $x_1>0$ or $x_2>0$.
In the first case, 
for some sufficiently small $\epsilon>0$, 
the vector $(x_1-\epsilon,x_2,0,x_4,x_5+\epsilon)$ satisfies the constraints, and
\begin{eqnarray*}
f(x_1-\epsilon,x_2,0,x_4,x_5+\epsilon)-f(x_1,x_2,0,x_4,x_5)
&=& -\left(\frac{x_1}{2}+x_2+x_4+\frac{3x_5}{2}\right)\epsilon-\frac{1}{2}\epsilon^2< 0 
\end{eqnarray*}
yields a contradiction to the choice of $(x_1,x_2,x_3,x_4,x_5)$.
In the second case, 
for some sufficiently small $\epsilon>0$, 
the vector $(x_1,x_2-\epsilon,0,x_4,x_5+\epsilon)$ satisfies the constraints, and
\begin{eqnarray*}
f(x_1,x_2-\epsilon,0,x_4,x_5+\epsilon)-f(x_1,x_2,0,x_4,x_5)
&=& -\left(\frac{x_1}{2}+x_2+\frac{x_4}{2}+x_5\right)\epsilon< 0 
\end{eqnarray*}
yields a contradiction to the choice of $(x_1,x_2,x_3,x_4,x_5)$.
Hence, we have $$3x_1+x_2=x_4+3x_5.$$
Note that 
\begin{eqnarray}
4x_1+2x_2&=&(x_1+x_2)+(3x_1+x_2)=x_1+x_2+x_4+3x_5\geq 1,\mbox{ which implies}\nonumber \\
\frac{x_1}{2}+x_2&\geq &\frac{1}{8}(4x_1+2x_2)\geq \frac{1}{8},\nonumber \\
x_1+x_2&\geq &\frac{1}{4}(4x_1+2x_2)\geq \frac{1}{4},\nonumber \\
x_4&\leq &1-(x_1+x_2)\leq \frac{3}{4},\mbox{ and, hence,}\nonumber \\
\frac{x_1}{2}+x_2-\frac{x_4}{8}&\geq &\frac{1}{32}.\label{e32}
\end{eqnarray}
If $x_2>0$, then, for some sufficiently small $\epsilon>0$, 
the vector $(x_1+\epsilon,x_2-2\epsilon,0,x_4+\epsilon,x_5)$ satisfies the constraints, and
\begin{eqnarray*}
f(x_1+\epsilon,x_2-2\epsilon,0,x_4+\epsilon,x_5)-f(x_1,x_2,0,x_4,x_5)
&=& -\left(\frac{x_1}{2}+x_2-\frac{x_4}{8}\right)\epsilon+\frac{13}{16}\epsilon^2\\
&\stackrel{(\ref{e32})}{\leq} & -\frac{1}{32}\epsilon+\frac{13}{16}\epsilon^2< 0
\end{eqnarray*}
yields a contradiction to the choice of $(x_1,x_2,x_3,x_4,x_5)$.
Hence, we have $$x_2=0.$$
Note that necessarily $x_1>0$.

If $x_5>0$, then, for some sufficiently small $\epsilon>0$, 
the vector $\left(x_1-\frac{\epsilon}{2},0,0,x_4+\frac{3\epsilon}{2},x_5-\epsilon\right)$ satisfies the constraints, and
\begin{eqnarray*}
f\left(x_1-\frac{\epsilon}{2},0,0,x_4+\frac{3\epsilon}{2},x_5-\epsilon\right)-f(x_1,0,0,x_4,x_5)
&=& -\left(\frac{x_1}{4}+\frac{5x_4}{16}\right)\epsilon-\frac{11}{64}\epsilon^2 < 0
\end{eqnarray*}
yields a contradiction to the choice of $(x_1,x_2,x_3,x_4,x_5)$.
Hence, we have $$x_5=0.$$
Now, $x_1+x_4=1$ and $3x_1=x_4$ imply that $x_1=\frac{1}{4}$ and $x_4=\frac{3}{4}$.
Since $f\left(\frac{1}{4},0,0,\frac{3}{4},0\right)=\frac{5}{256}$,
the desired result follows.
\end{proof}

\setcounter{claim}{0}

\begin{proof}[Proof of Theorem \ref{theorem2}]
Similarly as the proof of Theorem \ref{theorem1}, the current proof is by contradiction,
that is, we assume that $8$ divides $n$
and that $c$ is a zero-sum $\pm 1$-labeling of the edges of $K_n$ 
but that $K_n$ has no zero-sum $P_4$-factor.
In order to obtain some of the contradictions, 
we shall assume that $n$ is sufficiently large.
Exactly as in the proof of Theorem \ref{theorem1},
we obtain the existence of a $P_4$-factor $F$ of $K_n$ with $c(E(F))=2$.
For a path $P:u_1u_2u_3u_4$ in $K_n$, the {\it type} of $P$ is $(c(u_1u_2),c(u_2u_3),c(u_3u_4))$.
Up to symmetry, there are exactly the six following types:
$$(1,1,1),\,\,\,\,
(1,-1,1),\,\,\,\,
(1,1,-1),\,\,\,\,
(-1,1,-1),\,\,\,\,
(1,-1,-1),\,\,
(-1,-1,-1).$$
If $E_{\rm out}$ and $E_{\rm in}$ are two sets of edges of $K_n$ such that 
$(V(K_n),(E(F)\setminus E_{\rm out})\cup E_{\rm in})$
is a $P_4$-factor $F'$ of $K_n$ with $c(F')=0$,
then we say that the {\it edge-exchange $-E_{\rm out}+E_{\rm in}$} is {\it contradicting}.
Clearly, the existence of a contradicting edge-exchange yields a contradiction.

If $P:u_1u_2u_3u_4$ is a path in $F$,
$e$ is an edge of $P$ with $c(e)=1$, and 
$c(u_1u_4)=-1$, 
then the edge-exchange $-\{ e\}+\{ u_1u_4\}$ is contradicting.
This implies that 
\begin{eqnarray}\label{e1}
\mbox{\it $c(u_1u_4)=1$ for every path $u_1u_2u_3u_4$ in $F$ whose type is not $(-1,-1,-1)$.}
\end{eqnarray}
If $P:u_1u_2u_3u_4$ is a path in $F$ of type $(1,-1,1)$, 
then, by (\ref{e1}), the path $P':u_4u_1u_2u_3$ is of type $(1,1,-1)$.
Replacing $P$ within $F$ by $P'$ yields a $P_4$-factor $F'$ with $c(E(F'))=2$
and less paths of type $(1,-1,1)$.
This implies that we may assume that 
$$\mbox{\it no path in $F$ has type $(1,-1,1)$.}$$
If $P:u_1u_2u_3u_4$ is a path in $F$ whose type is $(1,1,1)$, $(1,1,-1)$, or $(-1,1,-1)$,
and $c(e)=-1$ for some edge $e$ in $\{ u_1u_3,u_2u_4\}$, 
then the edge-exchange $-\{ u_2u_3\}+\{ e\}$ is contradicting,
which implies that
\begin{eqnarray}\label{e2}
\mbox{\it $c(u_1u_3)=c(u_2u_4)=1$ for every path $u_1u_2u_3u_4$ in $F$ of type $(1,1,1)$, $(1,1,-1)$, or $(-1,1,-1)$.}
\end{eqnarray}
The observations (\ref{e1}) and (\ref{e2}) imply a number of very useful symmetries,
which we will exploit in order to reduce the number of cases.
If, for instance, $P$ is a path in $F$ of type $(1,1,1)$, 
then (\ref{e1}) and (\ref{e2}) imply that every $P_4$ in $K_n$ with vertex set $V(P)$ is of type $(1,1,1)$,
which implies complete symmetry between the vertices of $P$.
Similarly, if $P:u_1u_2u_3u_4$ is a path in $F$ of type $(1,1,-1)$, 
then (\ref{e1}) implies that $u_2u_1u_4u_3$ is of type $(1,1,-1)$,
which implies a symmetry between the vertices $u_1$ and $u_2$
as well as between the vertices $u_3$ and $u_4$.
If $P:u_1u_2u_3u_4$ is a path in $F$ of type $(1,-1,-1)$, 
then (\ref{e1}) implies that $u_1u_4u_3u_2$ is of type $(1,-1,-1)$,
which implies a symmetry between the vertices $u_2$ and $u_4$.
Finally, if $P:u_1u_2u_3u_4$ is a path in $F$ of type $(-1,1,-1)$, 
then (\ref{e1}) and (\ref{e2}) imply that 
$u_2u_1u_4u_3$,
$u_2u_1u_3u_4$, and
$u_1u_2u_4u_3$ are of type $(-1,1,-1)$,
which implies complete symmetry between the vertices of $P$.

Our overall goal is to derive the contradiction $c(E(K_n))>0$.
In order to achieve this, we investigate the edges between different paths in $F$.

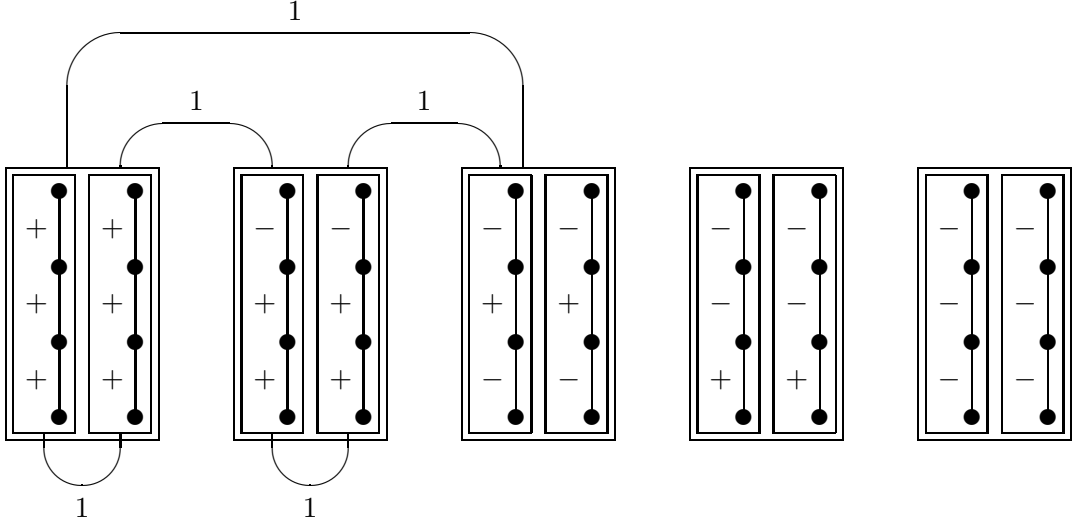
\begin{figure}[h]
\begin{center}
\unitlength 1mm 
\linethickness{0.4pt}
\ifx\plotpoint\undefined\newsavebox{\plotpoint}\fi 
\begin{picture}(145,71)(0,0)
\put(12,17){\circle*{2}}
\put(42,17){\circle*{2}}
\put(72,17){\circle*{2}}
\put(102,17){\circle*{2}}
\put(132,17){\circle*{2}}
\put(22,17){\circle*{2}}
\put(52,17){\circle*{2}}
\put(82,17){\circle*{2}}
\put(112,17){\circle*{2}}
\put(142,17){\circle*{2}}
\put(12,27){\circle*{2}}
\put(42,27){\circle*{2}}
\put(72,27){\circle*{2}}
\put(102,27){\circle*{2}}
\put(132,27){\circle*{2}}
\put(22,27){\circle*{2}}
\put(52,27){\circle*{2}}
\put(82,27){\circle*{2}}
\put(112,27){\circle*{2}}
\put(142,27){\circle*{2}}
\put(12,37){\circle*{2}}
\put(42,37){\circle*{2}}
\put(72,37){\circle*{2}}
\put(102,37){\circle*{2}}
\put(132,37){\circle*{2}}
\put(22,37){\circle*{2}}
\put(52,37){\circle*{2}}
\put(82,37){\circle*{2}}
\put(112,37){\circle*{2}}
\put(142,37){\circle*{2}}
\put(12,47){\circle*{2}}
\put(42,47){\circle*{2}}
\put(72,47){\circle*{2}}
\put(102,47){\circle*{2}}
\put(132,47){\circle*{2}}
\put(22,47){\circle*{2}}
\put(52,47){\circle*{2}}
\put(82,47){\circle*{2}}
\put(112,47){\circle*{2}}
\put(142,47){\circle*{2}}
\put(12,47){\line(0,-1){30}}
\put(42,47){\line(0,-1){30}}
\put(72,47){\line(0,-1){30}}
\put(102,47){\line(0,-1){30}}
\put(132,47){\line(0,-1){30}}
\put(22,47){\line(0,-1){30}}
\put(52,47){\line(0,-1){30}}
\put(82,47){\line(0,-1){30}}
\put(112,47){\line(0,-1){30}}
\put(142,47){\line(0,-1){30}}
\put(9,22){\makebox(0,0)[cc]{$+$}}
\put(39,22){\makebox(0,0)[cc]{$+$}}
\put(69,22){\makebox(0,0)[cc]{$-$}}
\put(99,22){\makebox(0,0)[cc]{$+$}}
\put(129,22){\makebox(0,0)[cc]{$-$}}
\put(19,22){\makebox(0,0)[cc]{$+$}}
\put(49,22){\makebox(0,0)[cc]{$+$}}
\put(79,22){\makebox(0,0)[cc]{$-$}}
\put(109,22){\makebox(0,0)[cc]{$+$}}
\put(139,22){\makebox(0,0)[cc]{$-$}}
\put(9,32){\makebox(0,0)[cc]{$+$}}
\put(39,32){\makebox(0,0)[cc]{$+$}}
\put(69,32){\makebox(0,0)[cc]{$+$}}
\put(99,32){\makebox(0,0)[cc]{$-$}}
\put(129,32){\makebox(0,0)[cc]{$-$}}
\put(19,32){\makebox(0,0)[cc]{$+$}}
\put(49,32){\makebox(0,0)[cc]{$+$}}
\put(79,32){\makebox(0,0)[cc]{$+$}}
\put(109,32){\makebox(0,0)[cc]{$-$}}
\put(139,32){\makebox(0,0)[cc]{$-$}}
\put(9,42){\makebox(0,0)[cc]{$+$}}
\put(39,42){\makebox(0,0)[cc]{$-$}}
\put(69,42){\makebox(0,0)[cc]{$-$}}
\put(99,42){\makebox(0,0)[cc]{$-$}}
\put(129,42){\makebox(0,0)[cc]{$-$}}
\put(19,42){\makebox(0,0)[cc]{$+$}}
\put(49,42){\makebox(0,0)[cc]{$-$}}
\put(79,42){\makebox(0,0)[cc]{$-$}}
\put(109,42){\makebox(0,0)[cc]{$-$}}
\put(139,42){\makebox(0,0)[cc]{$-$}}
\put(6,15){\framebox(8,34)[cc]{}}
\put(36,15){\framebox(8,34)[cc]{}}
\put(66,15){\framebox(8,34)[cc]{}}
\put(96,15){\framebox(8,34)[cc]{}}
\put(126,15){\framebox(8,34)[cc]{}}
\put(16,15){\framebox(8,34)[cc]{}}
\put(46,15){\framebox(8,34)[cc]{}}
\put(76,15){\framebox(8,34)[cc]{}}
\put(106,15){\framebox(8,34)[cc]{}}
\put(136,15){\framebox(8,34)[cc]{}}
\put(5,14){\framebox(20,36)[cc]{}}
\put(35,14){\framebox(20,36)[cc]{}}
\put(65,14){\framebox(20,36)[cc]{}}
\put(95,14){\framebox(20,36)[cc]{}}
\put(125,14){\framebox(20,36)[cc]{}}
\put(45,5){\makebox(0,0)[cc]{$1$}}
\put(15,5){\makebox(0,0)[cc]{$1$}}
\put(30,59){\makebox(0,0)[cc]{$1$}}
\put(43,71){\makebox(0,0)[cc]{$1$}}
\put(60,59){\makebox(0,0)[cc]{$1$}}
\put(45,15){\oval(10,14)[b]}
\put(15,15){\oval(10,14)[b]}
\put(30,50){\oval(20,12)[t]}
\put(60,50){\oval(20,12)[t]}
\put(43,50){\oval(60,36)[t]}
\end{picture}
\end{center}
\caption{Illustration of the statements in Claim \ref{claim1}.
The labels on the edges correspond to 
a lower bound on the fraction of all represented edges 
that have label $1$.}\label{fig1}
\end{figure}

\begin{claim}\label{claim1}
Let $P:u_1u_2u_3u_4$ and $Q:v_1v_2v_3v_4$ be two distinct paths in $F$ 
of type $t_P$ and $t_Q$, respectively.
If $(t_P,t_Q)$ is one of the following type pairs, then 
$c(e)=1$ for each of the $16$ edges $e$ of $K_n$ between $V(P)$ and $V(Q)$:
$((1,1,1),(1,1,-1))$,
$((1,1,1),(-1,1,-1))$,
$((1,1,-1),(-1,1,-1))$,
$((1,1,-1),(1,1,-1))$, and
$((1,1,1),(1,1,1))$.

See Figure \ref{fig1} for an illustration.
\end{claim}
\begin{proof}[Proof of Claim \ref{claim1}]
First, let $(t_P,t_Q)=((1,1,1),(1,1,-1))$.

If $c(e)=-1$ and $c(e')=1$ for 
some edge $e$ between $\{ u_3,u_4\}$ and $\{ v_3,v_4\}$ and
some edge $e'$ between $\{ u_1,u_2\}$ and $\{ v_1,v_2\}$, 
then the edge exchange $-\{ u_2u_3,v_2v_3\}+\{ e,e'\}$ is contradicting.
If $c(e)=-1$ and $c(e')=-1$ for 
\begin{itemize}
\item either 
some edge $e$ between $\{ u_3,u_4\}$ and $\{ v_3\}$ and
some edge $e'$ between $\{ u_1,u_2\}$ and $\{ v_1\}$
\item or
some edge $e$ between $\{ u_3,u_4\}$ and $\{ v_4\}$ and
some edge $e'$ between $\{ u_1,u_2\}$ and $\{ v_2\}$,
\end{itemize}
then the edge exchange $-\{ u_2u_3,v_1v_2,v_3v_4\}+\{ e,e',v_1v_4\}$ is contradicting.
If $c(e)=-1$ and $c(e')=-1$ for 
\begin{itemize}
\item either 
some edge $e$ between $\{ u_3,u_4\}$ and $\{ v_3\}$ and
some edge $e'$ between $\{ u_1,u_2\}$ and $\{ v_2\}$
\item or
some edge $e$ between $\{ u_3,u_4\}$ and $\{ v_4\}$ and
some edge $e'$ between $\{ u_1,u_2\}$ and $\{ v_1\}$,
\end{itemize}
then the edge exchange $-\{ u_2u_3,v_1v_2,v_2v_3,v_3v_4\}+\{ e,e',v_1v_3,v_2v_4\}$ is contradicting.
By the symmetry of the vertices of $P$, 
this implies that $c(e)=1$ for every edge $e$ between $V(P)$ and $\{ v_3,v_4\}$.

If $c(e)=-1$ for 
some edge $e$ between $\{ u_1,u_2\}$ and $\{ v_1,v_2\}$, 
then the edge exchange $-\{ u_2u_3,v_2v_3\}+\{ e,u_4v_4\}$ is contradicting.
By the symmetry of the vertices of $P$, 
this implies that $c(e)=1$ for every edge $e$ between $V(P)$ and $\{ v_1,v_2\}$.
Altogether, we obtain $c(e)=1$ for every edge $e$ between $V(P)$ and $V(Q)$.

\medskip

\noindent Next, let $(t_P,t_Q)=((1,1,1),(-1,1,-1))$.

If $c(u_2v_2)=c(u_4v_4)=-1$, 
then the edge exchange $-\{ u_1u_2,u_2u_3,v_2v_3,v_3v_4\}+\{ u_2v_2,u_4v_4,u_1u_3,v_1v_3\}$ is contradicting.
If $c(u_2v_2)=-1$ and $c(u_4v_4)=1$,
then the edge exchange $-\{ u_2u_3,v_2v_3\}+\{ u_2v_2,u_4v_4\}$ is contradicting.
By the symmetry of the vertices of $P$, 
this implies that $c(e)=1$ for every edge $e$ between $V(P)$ and $\{ v_2,v_3\}$.
If $c(u_1v_1)=-1$,
then the edge exchange $-\{ u_2u_3,v_2v_3\}+\{ e,u_3v_3\}$ is contradicting.
By the symmetry of the vertices of $P$, 
this implies that $c(e)=1$ for every edge $e$ between $V(P)$ and $\{ v_1,v_4\}$.
Altogether, we obtain $c(e)=1$ for every edge $e$ between $V(P)$ and $V(Q)$.

\medskip

\noindent Next, let $(t_P,t_Q)=((1,1,-1),(-1,1,-1))$.
 
If $c(u_2v_2)=c(u_4v_4)=-1$, 
then the edge exchange $-\{ u_1u_2,u_3u_4,v_2v_3\}+\{ u_2v_2,u_4v_4,u_1u_4\}$ is contradicting.
If $c(u_2v_2)=1$ and $c(u_4v_4)=-1$,
then the edge exchange $-\{ u_2u_3,v_2v_3\}+\{ u_2v_2,u_4v_4\}$ is contradicting.
By the symmetry between $u_3$ and $u_4$ as well as between the vertices of $Q$,
this implies that $c(e)=1$ for every edge $e$ between $\{ u_3,u_4\}$ and $V(Q)$.
If $c(u_2v_2)=-1$,
then the edge exchange $-\{ u_2u_3,v_2v_3\}+\{ u_2v_2,u_4v_4\}$ is contradicting.
By the symmetry between $u_1$ and $u_2$ as well as between the vertices of $Q$,
this implies that $c(e)=1$ for every edge $e$ between $\{ u_1,u_2\}$ and $V(Q)$.
Altogether, we obtain $c(e)=1$ for every edge $e$ between $V(P)$ and $V(Q)$.
 
\medskip

\noindent Next, let $(t_P,t_Q)=((1,1,-1),(1,1,-1))$.
 
If $c(u_2v_2)=c(u_4v_4)=-1$, 
then the edge exchange $-\{ u_1u_2,u_3u_4,v_2v_3\}+\{ u_2v_2,u_4v_4,u_1u_4\}$ is contradicting.
If $c(u_2v_2)=1$ and $c(u_4v_4)=-1$,
then the edge exchange $-\{ u_2u_3,v_2v_3\}+\{ u_2v_2,u_4v_4\}$ is contradicting.
By the symmetry between $u_3$ and $u_4$ as well as between $v_3$ and $v_4$,
this implies that $c(e)=1$ for every edge $e$ between $\{ u_3,u_4\}$ and $\{ v_3,v_4\}$.
If $c(u_2v_2)=-1$, 
then the edge exchange $-\{ u_2u_3,v_2v_3\}+\{ u_2v_2,u_4v_4\}$ is contradicting.
By the symmetry between $u_1$ and $u_2$ as well as between $v_1$ and $v_2$,
this implies that $c(e)=1$ for every edge $e$ between $\{ u_1,u_2\}$ and $\{ v_1,v_2\}$.
If $c(u_3v_2)=-1$ and $c(u_2v_3)=1$,
then the edge exchange $-\{ u_2u_3,v_2v_3\}+\{ u_3v_2,u_2v_3\}$ is contradicting.
If $c(u_3v_2)=c(u_2v_3)=-1$,
then the edge exchange $-\{ u_1u_2,u_2u_3,u_3u_4,v_2v_3\}+\{ u_3v_2,u_2v_3,u_1u_3,u_2u_4\}$ is contradicting.
By the symmetry between $u_3$ and $u_4$ as well as between $v_1$ and $v_2$,
this implies that $c(e)=1$ 
for every edge $e$ between $\{ u_3,u_4\}$ and $\{ v_1,v_2\}$
as well as
for every edge $e$ between $\{ u_1,u_2\}$ and $\{ v_3,v_4\}$.
Altogether, we obtain $c(e)=1$ for every edge $e$ between $V(P)$ and $V(Q)$.

\medskip

\noindent Finally, let $(t_P,t_Q)=((1,1,1),(1,1,1))$.

If $c(u_2v_2)=-1$ and $c(u_4v_4)=1$,
then the edge exchange $-\{ u_2u_3,v_2v_3\}+\{ u_2v_2,u_4v_4\}$ is contradicting.
By the symmetry between the vertices of $P$ as well as between the vertices of $Q$,
this implies that 
either $c(e)=-1$ for every edge $e$ between $V(P)$ and $V(Q)$
or $c(e)=1$ for every edge $e$ between $V(P)$ and $V(Q)$.
If $c(e)=-1$ for every edge $e$ between $V(P)$ and $V(Q)$,
and $R:w_1w_2w_3w_4$ is a path in $F$ of type $(1,1,-1)$, $(-1,1,-1)$, or $(1,-1,-1)$,
then the edge-exchange $-\{ u_1u_2,v_1v_2,w_3w_4\}+\{ u_1v_2,u_2v_1,w_1w_4\}$ is contradicting.
Hence, if $c(e)=-1$ for every edge $e$ between $V(P)$ and $V(Q)$,
then all paths in $F$ are of type $(1,1,1)$ or $(-1-,1,-1)$,
which implies the contradiction that $c(E(F))$ is a multiple of $3$.
Hence, 
$c(e)=1$ for every edge $e$ between $V(P)$ and $V(Q)$.
\end{proof}

\begin{figure}[h]
\begin{center}
\unitlength 1mm 
\linethickness{0.4pt}
\ifx\plotpoint\undefined\newsavebox{\plotpoint}\fi 
\begin{picture}(145,85)(0,0)
\put(12,26){\circle*{2}}
\put(42,26){\circle*{2}}
\put(72,26){\circle*{2}}
\put(102,26){\circle*{2}}
\put(132,26){\circle*{2}}
\put(22,26){\circle*{2}}
\put(52,26){\circle*{2}}
\put(82,26){\circle*{2}}
\put(112,26){\circle*{2}}
\put(142,26){\circle*{2}}
\put(12,36){\circle*{2}}
\put(42,36){\circle*{2}}
\put(72,36){\circle*{2}}
\put(102,36){\circle*{2}}
\put(132,36){\circle*{2}}
\put(22,36){\circle*{2}}
\put(52,36){\circle*{2}}
\put(82,36){\circle*{2}}
\put(112,36){\circle*{2}}
\put(142,36){\circle*{2}}
\put(12,46){\circle*{2}}
\put(42,46){\circle*{2}}
\put(72,46){\circle*{2}}
\put(102,46){\circle*{2}}
\put(132,46){\circle*{2}}
\put(22,46){\circle*{2}}
\put(52,46){\circle*{2}}
\put(82,46){\circle*{2}}
\put(112,46){\circle*{2}}
\put(142,46){\circle*{2}}
\put(12,56){\circle*{2}}
\put(42,56){\circle*{2}}
\put(72,56){\circle*{2}}
\put(102,56){\circle*{2}}
\put(132,56){\circle*{2}}
\put(22,56){\circle*{2}}
\put(52,56){\circle*{2}}
\put(82,56){\circle*{2}}
\put(112,56){\circle*{2}}
\put(142,56){\circle*{2}}
\put(12,56){\line(0,-1){30}}
\put(42,56){\line(0,-1){30}}
\put(72,56){\line(0,-1){30}}
\put(102,56){\line(0,-1){30}}
\put(132,56){\line(0,-1){30}}
\put(22,56){\line(0,-1){30}}
\put(52,56){\line(0,-1){30}}
\put(82,56){\line(0,-1){30}}
\put(112,56){\line(0,-1){30}}
\put(142,56){\line(0,-1){30}}
\put(9,31){\makebox(0,0)[cc]{$+$}}
\put(39,31){\makebox(0,0)[cc]{$+$}}
\put(69,31){\makebox(0,0)[cc]{$-$}}
\put(99,31){\makebox(0,0)[cc]{$+$}}
\put(129,31){\makebox(0,0)[cc]{$-$}}
\put(19,31){\makebox(0,0)[cc]{$+$}}
\put(49,31){\makebox(0,0)[cc]{$+$}}
\put(79,31){\makebox(0,0)[cc]{$-$}}
\put(109,31){\makebox(0,0)[cc]{$+$}}
\put(139,31){\makebox(0,0)[cc]{$-$}}
\put(9,41){\makebox(0,0)[cc]{$+$}}
\put(39,41){\makebox(0,0)[cc]{$+$}}
\put(69,41){\makebox(0,0)[cc]{$+$}}
\put(99,41){\makebox(0,0)[cc]{$-$}}
\put(129,41){\makebox(0,0)[cc]{$-$}}
\put(19,41){\makebox(0,0)[cc]{$+$}}
\put(49,41){\makebox(0,0)[cc]{$+$}}
\put(79,41){\makebox(0,0)[cc]{$+$}}
\put(109,41){\makebox(0,0)[cc]{$-$}}
\put(139,41){\makebox(0,0)[cc]{$-$}}
\put(9,51){\makebox(0,0)[cc]{$+$}}
\put(39,51){\makebox(0,0)[cc]{$-$}}
\put(69,51){\makebox(0,0)[cc]{$-$}}
\put(99,51){\makebox(0,0)[cc]{$-$}}
\put(129,51){\makebox(0,0)[cc]{$-$}}
\put(19,51){\makebox(0,0)[cc]{$+$}}
\put(49,51){\makebox(0,0)[cc]{$-$}}
\put(79,51){\makebox(0,0)[cc]{$-$}}
\put(109,51){\makebox(0,0)[cc]{$-$}}
\put(139,51){\makebox(0,0)[cc]{$-$}}
\put(6,24){\framebox(8,34)[cc]{}}
\put(36,24){\framebox(8,34)[cc]{}}
\put(66,24){\framebox(8,34)[cc]{}}
\put(96,24){\framebox(8,34)[cc]{}}
\put(126,24){\framebox(8,34)[cc]{}}
\put(16,24){\framebox(8,34)[cc]{}}
\put(46,24){\framebox(8,34)[cc]{}}
\put(76,24){\framebox(8,34)[cc]{}}
\put(106,24){\framebox(8,34)[cc]{}}
\put(136,24){\framebox(8,34)[cc]{}}
\put(5,23){\framebox(20,36)[cc]{}}
\put(35,23){\framebox(20,36)[cc]{}}
\put(65,23){\framebox(20,36)[cc]{}}
\put(95,23){\framebox(20,36)[cc]{}}
\put(125,23){\framebox(20,36)[cc]{}}
\put(90,67){\makebox(0,0)[cc]{$\frac{1}{2}$}}
\put(74,14){\makebox(0,0)[cc]{$\frac{1}{2}$}}
\put(114,8){\makebox(0,0)[cc]{$\frac{1}{2}$}}
\put(107,78){\makebox(0,0)[cc]{$\frac{1}{2}$}}
\put(75,23){\oval(50,24)[b]}
\put(90,23){\oval(90,36)[b]}
\put(60,59){\oval(90,24)[t]}
\put(75,59){\oval(130,46)[t]}
\put(74,85){\makebox(0,0)[cc]{$\frac{3}{4}$}}
\put(60,74){\makebox(0,0)[cc]{$\frac{3}{4}$}}
\put(107,59){\oval(60,32)[t]}
\put(90,59){\oval(20,10)[t]}
\end{picture}
\end{center}
\caption{Illustration of the statements in Claim \ref{claim2} and Claim \ref{claim3}.
Again, the labels on the edges correspond to 
a lower bound on the fraction of all represented edges 
that have label $1$.}\label{fig2}
\end{figure}
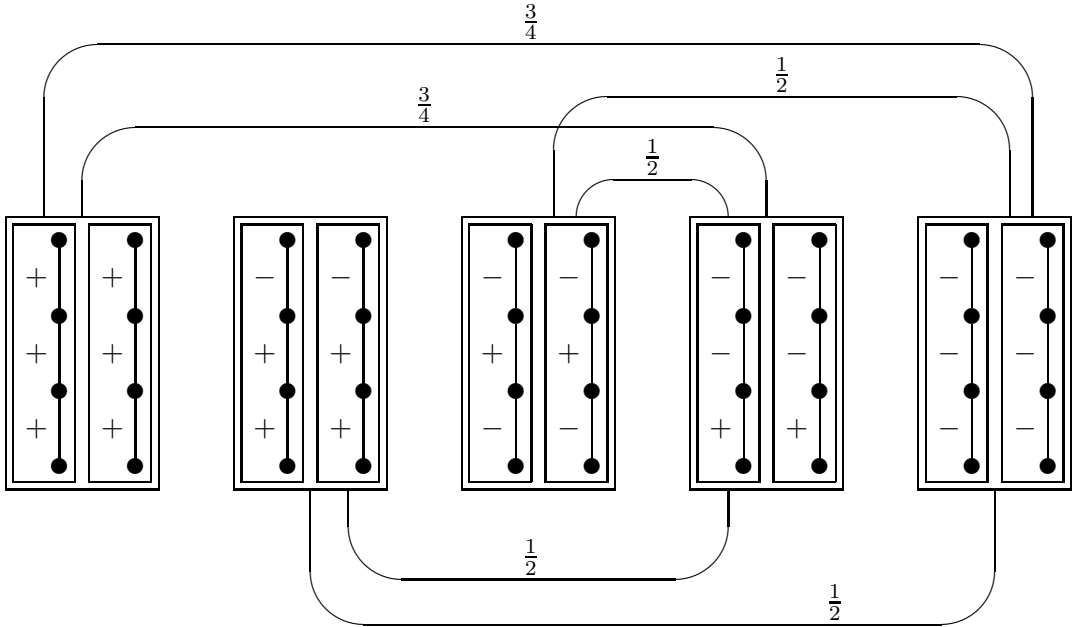

\begin{claim}\label{claim2}
Let $P:u_1u_2u_3u_4$ and $Q:v_1v_2v_3v_4$ be two distinct paths in $F$ 
of types $t_P$ and $t_Q$, respectively.
If $(t_P,t_Q)$ is 
$((1,1,1),(-1,-1,-1))$ 
or $((1,1,1),(1,-1,-1))$,
then $c(e)=1$ for at least $12$ of the $16$ edges $e$ of $K_n$ between $V(P)$ and $V(Q)$.

See Figure \ref{fig2} for an illustration.
\end{claim}
\begin{proof}[Proof of Claim \ref{claim2}]
First, let $(t_P,t_Q)=((1,1,1),(-1,-1,-1))$.

If $c(e)=c(e')=-1$ for some edge $e$ between $\{ u_1,u_2\}$ and $\{ v_1,v_2\}$
and some edge $e'$ between $\{ u_3,u_4\}$ and $\{ v_3,v_4\}$,
then the edge exchange $-\{ u_2u_3,v_2v_3\}+\{ e,e'\}$ is contradicting.
If $c(u_1v_1)=c(u_2v_1)=c(u_2v_2)=-1$ and $c(u_3v_3)=1$,
then the edge exchange $-\{ u_1u_2,u_2u_3,v_1v_2,v_2v_3\}+\{ u_1v_1,u_2v_1,u_2v_2,u_3v_3\}$ is contradicting.
By symmetry, this implies $c(e)=-1$ for at most $2$ of the $8$ edges $e$
between either $\{ u_1,u_2\}$ and $\{ v_1,v_2\}$
or $\{ u_3,u_4\}$ and $\{ v_3,v_4\}$.
By the symmetry between the vertices of $P$, this also implies that 
$c(e)=-1$ for at most $2$ of the $8$ edges $e$
between either $\{ u_1,u_2\}$ and $\{ v_3,v_4\}$
or $\{ u_3,u_4\}$ and $\{ v_1,v_2\}$.
Altogether, we obtain that $c(e)=-1$ for at most $4$ edges $e$ between $V(P)$ and $V(Q)$.

\medskip

\noindent Next, let $(t_P,t_Q)=((1,1,1),(1,-1,-1))$.

If $c(u_1v_1)=c(u_2v_2)=-1$,
then the edge exchange $-\{ u_1u_2,u_3u_4,v_1v_2,v_3v_4\}+\{ u_1v_1,u_2v_2,u_1u_4,v_1v_4\}$ is contradicting.
If $c(u_1v_1)=-1$ and $c(u_2v_2)=1$,
then the edge exchange $-\{ u_1u_2,u_2u_3,v_1v_2\}+\{ u_1v_1,u_2v_2,u_1u_3\}$ is contradicting.
By the symmetry between the vertices of $P$, this implies that 
$c(e)=1$ for every edge $e$ between $V(P)$ and $v_1$.
If $c(u_4v_4)=-1$,
then the edge exchange $-\{ u_3u_4,v_1v_2\}+\{ u_1v_1,u_4v_4\}$ is contradicting.
By the symmetry between the vertices of $P$ and the symmetry between $v_2$ and $v_4$, 
this implies that 
$c(e)=1$ for every edge $e$ between $V(P)$ and $\{ v_2,v_4\}$.
Altogether, we obtain that $c(e)=-1$ for at most $4$ edges $e$ between $V(P)$ and $V(Q)$.
\end{proof}

\begin{claim}\label{claim3}
Let $P:u_1u_2u_3u_4$ and $Q:v_1v_2v_3v_4$ be two distinct paths in $F$ 
of types $t_P$ and $t_Q$, respectively.
If $(t_P,t_Q)$ is 
$((1,1,-1),(1,-1,-1))$,
$((-1,1,-1),(1,-1,-1))$,
$((1,1,-1),(-1,-1,-1))$, or
$((-1,1,-1),(-1,-1,-1))$,
then $c(e)=1$ for at least $8$ of the $16$ edges $e$ of $K_n$ between $V(P)$ and $V(Q)$.

See Figure \ref{fig2} for an illustration.
\end{claim}
\begin{proof}[Proof of Claim \ref{claim3}]
First, let $(t_P,t_Q)=((1,1,-1),(1,-1,-1))$.

If $c(u_3v_1)=c(u_2v_4)=-1$,
then the edge exchange $-\{ u_2u_3,v_2v_3\}+\{ u_3v_1,u_2v_4\}$ is contradicting.
If $c(u_3v_1)=-1$ and $c(u_2v_4)=1$,
then the edge exchange $-\{ u_1u_2,u_2u_3,v_1v_2\}+\{ u_3v_1,u_2v_4,u_1u_4\}$ is contradicting.
If $c(u_1v_2)=1$ and $c(u_2v_1)=-1$,
then the edge exchange $-\{ u_1u_2,v_1v_2\}+\{ u_1v_2,u_2v_1\}$ is contradicting.
If $c(u_1v_2)=c(u_2v_1)=-1$,
then the edge exchange $-\{ u_1u_2,u_2u_3,v_1v_2,v_2v_3\}+\{ u_1v_2,u_2v_1,u_1u_4,v_1v_4\}$ is contradicting.
By the symmetry between $u_1$ and $u_2$ as well as between $u_3$ and $u_4$, 
it follows that $c(e)=1$ for every edge $e$ between $V(P)$ and $v_1$.
If $c(u_1v_2)=-1$,
then the edge exchange $-\{ u_1u_2,v_1v_2\}+\{ u_1v_2,u_2v_1\}$ is contradicting.
By the symmetry between $u_1$ and $u_2$ as well as between $v_2$ and $v_4$,
it follows that $c(e)=1$ for every edge $e$ between $\{ u_1,u_2\}$ and $\{ v_2,v_4\}$.
Altogether, we obtain that $c(e)=1$ for at least $8$ edges $e$ between $V(P)$ and $V(Q)$.

\medskip

\noindent Next, let $(t_P,t_Q)=((-1,1,-1),(1,-1,-1))$.

If $c(e)=c(e')=-1$ for 
some edge $e$ between $\{ u_1,u_2\}$ and $v_2$
and 
some edge $e'$ between $\{ u_3,u_4\}$ and $v_4$,
then the edge exchange $-\{ u_2u_3,v_1v_2,v_3v_4\}+\{ e,e',v_1v_4\}$ is contradicting.
If $c(e)=c(e')=-1$ for 
some edge $e$ between $\{ u_1,u_2\}$ and $v_1$
and 
some edge $e'$ between $\{ u_3,u_4\}$ and $v_3$,
then the edge exchange $-\{ u_2u_3,v_2v_3\}+\{ e,e'\}$ is contradicting.
By the symmetry between the vertices of $P$,
this implies that $c(e)=1$ for at least $8$ edges $e$ between $V(P)$ and $V(Q)$.

\medskip

\noindent Next, let $(t_P,t_Q)=((1,1,-1),(-1,-1,-1))$.

If $c(e)=c(e')=-1$ for 
some edge $e$ between $\{ u_1,u_2\}$ and $\{ v_1,v_2\}$
and 
some edge $e'$ between $\{ u_3,u_4\}$ and $\{ v_3,v_4\}$,
then the edge exchange $-\{ u_2u_3,v_2v_3\}+\{ e,e'\}$ is contradicting.
By symmetry,
this implies that $c(e)=1$ for at least $8$ edges $e$ between $V(P)$ and $V(Q)$.

\medskip

\noindent Finally, let $(t_P,t_Q)=((-1,1,-1),(-1,-1,-1))$.

If $c(e)=c(e')=-1$ for 
some edge $e$ between $\{ u_1,u_2\}$ and $\{ v_1,v_2\}$
and 
some edge $e'$ between $\{ u_3,u_4\}$ and $\{ v_3,v_4\}$,
then the edge exchange $-\{ u_2u_3,v_2v_3\}+\{ e,e'\}$ is contradicting.
By symmetry,
this implies that $c(e)=1$ for at least $8$ edges $e$ between $V(P)$ and $V(Q)$.
\end{proof}

\begin{claim}\label{claim4}
If $P:u_1u_2u_3u_4$ and $Q:v_1v_2v_3v_4$ are two distinct paths in $F$ 
of types $(1,-1,-1)$ and $(-1,-1,-1)$, respectively,
then $c(e)=1$ for at least $4$ of the $16$ edges $e$ of $K_n$ between $V(P)$ and $V(Q)$.
\end{claim}
\begin{proof}[Proof of Claim \ref{claim4}]
If $c(e)=c(e')=-1$ for 
some edge $e$ between $u_1$ and $\{ v_1,v_3\}$
and 
some edge $e'$ between $\{ u_2,u_4\}$ and $v_4$,
then the edge exchange $-\{ u_1u_2,v_3v_4\}+\{ e,e'\}$ is contradicting.
If $c(e)=c(e')=-1$ for 
some edge $e$ between $u_1$ and $\{ v_2,v_4\}$
and 
some edge $e'$ between $\{ u_2,u_4\}$ and $v_1$,
then the edge exchange $-\{ u_1u_2,v_1v_2\}+\{ e,e'\}$ is contradicting.
By symmetry,
this implies that $c(e)=1$ for at least $4$ edges $e$ between $V(P)$ and $V(Q)$.
\end{proof}
Trivially, the $P_4$-factor $F$ of $K_n$ consists of exactly $\frac{n}{4}$ paths.

Let $x_1,\ldots,x_5\in [0,1]$ be such that 
\begin{itemize}
\item $F$ contains $\frac{x_1n}{4}$ paths of type $(1,1,1)$,
\item $F$ contains $\frac{x_2n}{4}$ paths of type $(1,1,-1)$,
\item $F$ contains $\frac{x_3n}{4}$ paths of type $(-1,1,-1)$,
\item $F$ contains $\frac{x_4n}{4}$ paths of type $(1,-1,-1)$, and
\item $F$ contains $\frac{x_5n}{4}$ paths of type $(-1,-1,-1)$.
\end{itemize}
Since 
\begin{eqnarray}
\frac{n}{4} &=& \frac{x_1n}{4}+\frac{x_2n}{4}+\frac{x_3n}{4}+\frac{x_4n}{4}+\frac{x_5n}{4},\mbox{ and}\nonumber\\
c(E(F))&=&\frac{3x_1n}{4}+\frac{x_2n}{4}-\frac{x_3n}{4}-\frac{x_4n}{4}-\frac{3x_5n}{4}=2,\label{e3}
\end{eqnarray}
we obtain 
\begin{eqnarray}\label{e4}
x_1+x_2+x_3+x_4+x_5 & = & 1,\mbox{ and}\label{e4}\\
3x_1+x_2-x_3-x_4-3x_5 & \geq & 0.\label{e5}
\end{eqnarray}
The following claim is the first point within the current proof
that requires $n$ to be sufficiently large.

\begin{claim}\label{claim5}
$F$ contains at least $3$ paths that are of type $(1,1,-1)$, $(-1,1,-1)$, or $(1,-1,-1)$.
\end{claim}
\begin{proof}[Proof of Claim \ref{claim5}]
Suppose, for a contradiction, that $F$ contains at most $2$ paths that are not of type $(1,1,1)$ or $(-1,-1,-1)$,
that is, $\frac{(x_2-x_3-x_4)n}{4}\leq \frac{(x_2+x_3+x_4)n}{4}\leq 2$.
By (\ref{e3}), this implies $x_1\geq x_5$.
Since $n'=n-(x_2+x_3+x_4)n$ is the number of vertices on paths of type $(1,1,1)$ or $(-1,-1,-1)$,
we have $n'\geq n-8$.
Let $n_+=x_1n$ and $n_-=x_5n$.
Clearly, we have $n'=n_++n_-$ and $n_+\geq n_-$.
By (\ref{e1}), (\ref{e2}), Claim \ref{claim1}, and Claim \ref{claim2}, we obtain
\begin{eqnarray*}
c(E(K_n)) & \geq & {n_+\choose 2}-{n_-\choose 2}+\frac{1}{2}n_+n_--8n\\
&=& \frac{1}{2}n_+^2-\frac{1}{2}n_-^2+\frac{1}{2}n_+n_--\frac{1}{2}n_++\frac{1}{2}n_--8n\\
&\geq & \frac{1}{2}\left(n_+^2-n_-^2+n_+n_-\right)-\frac{17}{2}n\\
&=& \frac{1}{2}\left(\left(\frac{n'}{2}\right)^2+
\underbrace{\left(n_+-\frac{n'}{2}\right)\left(\frac{5n'}{2}-n_+\right)}_{\geq 0}\right)
-\frac{17}{2}n\\
&\geq & \frac{(n-8)^2}{8}-\frac{17}{2}n\\
&> & \frac{1}{8}n^2-\frac{21}{2}n,
\end{eqnarray*}
which is positive for $n\geq 84$.
Hence, for sufficiently large $n$, we obtain the desired contradiction.
\end{proof}

\begin{claim}\label{claim6}
If $P$ is a path in $F$ of type $(1,1,-1)$, $(-1,1,-1)$, or $(1,-1,-1)$, 
then there is no $P_4$-factor $F'$ in $K_n$ with $c(E(F'))=-2$ that contains $P$.
\end{claim}
\begin{proof}[Proof of Claim \ref{claim6}]
If $P$ is $u_1u_2u_3u_4$, then, by (\ref{e1}), 
$(V(K_n),(E(F')\setminus \{ u_3u_4\})\cup \{ u_1u_4\})$
is a $P_4$-factor $F''$ in $K_n$ with $c(E(F''))=0$,
which is a contradiction.
\end{proof}

\begin{claim}\label{claim7}
If $P:u_1u_2u_3u_4$ and $Q:v_1v_2v_3v_4$ are two distinct paths in $F$ 
of type $(-1,1,-1)$,
then $c(e)=1$ for each of the $16$ edges $e$ between $V(P)$ and $V(Q)$.
\end{claim}
\begin{proof}[Proof of Claim \ref{claim7}]
Suppose, for a contradiction, that $c(e)=-1$ for some edge $e$ between $V(P)$ and $V(Q)$.
By symmetry, we may assume that $e=u_1v_1$.
If $c(u_3v_3)=1$,
then the edge exchange $-\{ u_2u_3,v_2v_3\}+\{ u_1v_1,u_3v_3\}$ is contradicting.
This implies $c(u_3v_3)=-1$ and, consequently, 
$(V(K_n),(E(F)\setminus \{ u_2u_3,v_2v_3\})\cup \{ u_1v_1,u_3v_3\})$
is a $P_4$-factor $F'$ of $K_n$ with $c(E(F'))=-2$
that shares all but the two paths $P$ and $Q$ with $F$.
Together with Claim \ref{claim5} and Claim \ref{claim6}, 
this clearly yields a contradiction.
\end{proof}
 
\begin{claim}\label{claim8}
If $P:u_1u_2u_3u_4$ and $Q:v_1v_2v_3v_4$ are two distinct paths in $F$ 
of type $(1,-1,-1)$,
then $c(e)=1$ for at least $5$ of the $16$ edges $e$ of $K_n$ between $V(P)$ and $V(Q)$.
\end{claim}
\begin{proof}[Proof of Claim \ref{claim8}]
If $c(u_2v_1)=-1$,
then 
$(V(K_n),(E(F)\setminus \{ u_1u_2,v_1v_2\})\cup \{ u_2v_1,u_1v_2\})$
is a $P_4$-factor $F'$ of $K_n$ with $c(E(F'))\in \{ 0,-2\}$.
If $c(u_1v_1)=c(u_3v_1)=-1$,
then 
$(V(K_n),(E(F)\setminus \{ u_1u_2,u_2u_3,v_1v_2\})\cup \{ u_1v_1,u_3v_1,u_2v_2\})$
is a $P_4$-factor $F'$ of $K_n$ with $c(E(F'))\in \{ 0,-2\}$.
By Claim \ref{claim5} and Claim \ref{claim6},
both cases yield a contradiction.
By symmetry, 
it follows that $c(e)=1$
for all edges $e$ in $\{ u_2v_1,u_4v_1,u_1v_2,u_1v_4\}$
as well as at least one edge $e$ in $\{ u_1v_1,u_3v_1\}$.
Altogether, there are at least $5$ edges $e$ with $c(e)=1$
between $V(P)$ and $V(Q)$.
\end{proof}
We can now use the claims to lower bound $c(E(K_n))$.
If $E_1$ is the set of edges of $K_n$ between different paths of type $(1,1,1)$ in $F$, 
then $|E_1|=16{\frac{x_1n}{4}\choose 2}$.
By Claim \ref{claim1}, all these edges $e$ satisfy $c(e)=1$, and, thus, 
$$c(E_1)=|E_1|=\frac{x_1^2n^2}{2}-2x_1n.$$
Similarly, if $E_4$ is the set of edges of $K_n$ between different paths of type $(1,-1,-1)$ in $F$, 
then Claim \ref{claim8} implies that the $16$ edges between any two such paths contribute at least
$5-11$ to the total sum of the edge labels, and, thus, 
$$c(E_4)\geq (5-11){\frac{x_4n}{4}\choose 2}=-\left(\frac{3x_4^2n^2}{16}-\frac{3x_4n}{4}\right).$$
If $E_{1,5}$ and $E_{2,5}$ are the sets of edges of $K_n$ 
between the paths of type $(1,1,1)$ and the paths of type $(-1,-1,-1)$
and 
between the paths of type $(1,1,-1)$ and the paths of type $(-1,-1,-1)$, respectively,
then Claim \ref{claim2} and Claim \ref{claim3} imply
\begin{eqnarray*}
c(E_{1,5}) & \geq & (12-4)\frac{x_1n}{4}\frac{x_5n}{4}=\frac{x_1x_5n^2}{2}\mbox{ and }\\
c(E_{2,5}) & \geq & (8-8)\frac{x_2n}{4}\frac{x_5n}{4}=0.
\end{eqnarray*}
Note that there are exactly $\frac{6n}{4}$ edges in $K_n$ 
that each have both their endpoints in one of the paths in $F$.
Therefore, combining the estimates implied by Claims 
\ref{claim1},
\ref{claim2},
\ref{claim3},
\ref{claim4},
\ref{claim7}, and
\ref{claim8}, we obtain that 
\begin{eqnarray*}
c(E(K_n)) & \geq & 
\left(\frac{x_1^2n^2}{2}-2x_1n\right)
+\left(\frac{x_2^2n^2}{2}-2x_2n\right)
+\left(\frac{x_3^2n^2}{2}-2x_3n\right)\\
&&
-\left(\frac{3x_4^2n^2}{16}-\frac{3x_4n}{4}\right)
-\left(\frac{x_5^2n^2}{2}-2x_5n\right)\\
&&
+x_1x_2n^2
+x_1x_3n^2
+x_2x_3n^2
+\frac{x_1x_4n^2}{2}
+\frac{x_1x_5n^2}{2}
-\frac{x_4x_5n^2}{2}
-\frac{6n}{4}\\
&=& 
\left(
\underbrace{\frac{x_1^2}{2}+\frac{x_2^2}{2}+\frac{x_3^2}{2}-\frac{3x_4^2}{16}-\frac{x_5^2}{2}
+x_1x_2+x_1x_3+x_2x_3+\frac{x_1x_4}{2}+\frac{x_1x_5}{2}-\frac{x_4x_5}{2}}_{=f(x_1,x_2,x_3,x_4,x_5)
\,\,as\,\,in\,\,Lemma\,\,\ref{lemma1}}
\right)n^2\\
&& -\left(
\underbrace{2x_1+2x_2+2x_3-\frac{3}{4}x_4-2x_5+\frac{3}{2}}_{
\leq 2(x_1+x_2+x_3)+\frac{3}{2}
\stackrel{(\ref{e4})}{\leq}\frac{7}{2}}
\right)n\\
& \geq & \frac{5}{256}n^2-\frac{7}{2}n,
\end{eqnarray*}
where the final inequality follows using Lemma \ref{lemma1}, (\ref{e4}), and (\ref{e5}).
For sufficiently large $n$, we obtain the final contradiction $c(E(K_n))>0$,
which completes the proof.
\end{proof}

\end{document}